\documentclass{birkjour}

\usepackage{amsmath, amsthm, amsfonts, amssymb}
\usepackage{enumerate}
\usepackage{paralist}

\usepackage{pstricks}
\usepackage[normalem]{ulem}

 \newtheorem{thm}{Theorem}[section]
 \newtheorem{cor}[thm]{Corollary}
 \newtheorem{lem}[thm]{Lemma}
 \newtheorem{prop}[thm]{Proposition}
 \theoremstyle{definition}
 \newtheorem{dfn}[thm]{Definition}
 \theoremstyle{remark}
 \newtheorem{rem}[thm]{Remark}
 \newtheorem{example}[thm]{Example}
 
 \numberwithin{equation}{section}

\newcommand{\C}{\mathbb{C}}

\newcommand{\R}{\mathbb{R}}

\DeclareMathOperator{\re}{Re}

\newcommand{\set}[1]{\{#1\}}
\newcommand{\enorm}{\|\cdot\|}

\newcommand{\ip}[1]{\langle #1 \rangle}

\begin{document}
\title[An orthogonality relation in complex normed spaces]%
{An orthogonality relation in complex normed spaces based on norm derivatives}
\author{S.~M.~Enderami}

\address{%
School of Mathematics and Computer Sciences,\\
Damghan University, P.O.BOX 36715-364, Damghan, Iran.}
\email{sm.enderami@std.du.ac.ir}

\author{M.~Abtahi}
\address{%
School of Mathematics and Computer Sciences,\\
Damghan University, P.O.BOX 36715-364, Damghan, Iran.}
\email{abtahi@du.ac.ir}

\thanks{Corresponding author (M. Abtahi) Tel: +982335220092, Email: abtahi@du.ac.ir}

\author{A.~Zamani}
\address{%
School of Mathematics and Computer Sciences,\\
Damghan University, P.O.BOX 36715-364, Damghan, Iran.}
\email{zamani.ali85@yahoo.com}

\author{P.~W{\'o}jcik}
\address{%
Institute of Mathematics, Pedagogical University of Cracow,
Podchor\b{a}\.{z}ych~2, 30-084 Krak\'{o}w, Poland.}
\email{pawel.wojcik@up.krakow.pl}

\subjclass{Primary 46B20; Secondary 46C50, 47B49.}

\keywords{Norm derivatives, orthogonality, orthogonality preserving mappings, inner product space.}
\begin{abstract}
Let $X$ be a complex normed space. Based on the right norm derivative $\rho_{_{+}}$,
we define a mapping $\rho_{_{\infty}}$ by
\begin{equation*}
\rho_{_{\infty}}(x,y) = \frac1\pi\int_0^{2\pi}e^{i\theta}\rho_{_{+}}(x,e^{i\theta}y)d\theta
\quad(x,y\in X).
\end{equation*}
The mapping $\rho_{_{\infty}}$ has a good response to some geometrical properties of
$X$. For instance, we prove that $\rho_{_{\infty}}(x,y)=\rho_{_{\infty}}(y,x)$
for all $x, y \in X$ if and only if $X$ is an inner product space.
In addition, we define a $\rho_{_{\infty}}$-orthogonality in $X$ and show that a linear mapping
preserving $\rho_{_{\infty}}$-orthogonality has to be a scalar multiple
of an isometry. A number of challenging problems in the geometry of complex normed spaces are also discussed.
\end{abstract}
\maketitle

\section{Introduction and preliminaries}
Let $(X,\|\!\cdot\!\|)$ be a complex normed space and $X^*$ its topological dual space.
Let $S_X$ denote the unit sphere of $X$. For a fixed point $x\in X$, by $J(x)$
we mean the set of supporting functionals at $x$, i.e.,
\begin{equation*}
J(x)=\left\{f\in S_{X^*}: \,f(x) = \|x\|\right\}.
\end{equation*}
The Hahn--Banach theorem guarantees that the set $J(x)$ is always nonempty.
An element $x\in X\setminus\{0\}$ is said to be a smooth point if $J(x)$ is a singleton.
A normed space $X$ is called smooth if every $x\in X\setminus\{0\}$ is a smooth point.
Given two elements $x, y\in X$, let ${\rm conv}\{x,y\}$ denote the closed line segment joining $x$ and $y$.
The normed space $X$ is called rotund or strictly convex if
\begin{equation*}
\forall_{x, y \in S_X}\,\, {\rm conv}\{x,y\} \subseteq S_X \Longrightarrow x=y.
\end{equation*}
By $\mathcal{R}(X)$ we denote (see \cite{St.Wo}) the length of the ``longest'' line segment lying
on the unit sphere; more precisely,
\begin{equation*}
\mathcal{R}(X):=\sup\left\{\|x-y\|: {\rm conv}\{x,y\}\subseteq S_X\right\}.
\end{equation*}
Clearly, $\mathcal{R}(X)=0$ if and only if $X$ is rotund.
Due to G.~Lumer \cite{Lumer} and J.~R.~Giles \cite{giles} in every normed space $(X,\|\!\cdot\!\|)$,
there exists a mapping $[\cdot, \cdot] : X \times X \to \mathbb{C}$ satisfying the following properties:
\begin{enumerate}[]
\item (sip1) \label{sip1.perp}
$[\alpha x + y, z] = \alpha [x, z] + [y, z]$, for all $x, y, z \in X$ and $\alpha, \beta\in\mathbb{C}$,
\item (sip2) \label{sip2.perp}
$[x, \alpha y] = \overline{\alpha}[x, y]$, for all $x, y \in X$ and $\alpha\in\mathbb{C}$,
\item (sip3) \label{sip3.perp}
$[x, x] = \|x\|^2$, for all $x\in X$,
\item (sip4)  \label{sip4.perp}
$|[x, y]| \leq \|x\|\|y\|$, for all $x, y \in X$.
\end{enumerate}
Such a mapping is called a semi-inner product (s.i.p.) in $X$
(generating the norm $\|\!\cdot\!\|$). There may exist infinitely many different semi-inner products in
$X$. There is a unique one if and only if $X$ is smooth.
If $X$ is an inner product space, the only s.i.p. on $X$ is the inner product itself.
We define two mappings $\rho_{_{+}}, \rho_{_{-}}:X\times X\rightarrow\mathbb{R}$ by the formulas
\begin{align*}
\rho_{_{\pm}}(x,y):=\lim_{t\rightarrow0^{\pm}}\frac{\|x+ty\|^2-\|x\|^2}{2t}=\|x\|\lim_{t\rightarrow0^{\pm}}\frac{\|x+ty\|-\|x\|}{t}.
\end{align*}
The convexity of the norm yields that the above definitions are meaningful. These mappings
are called the norm derivatives and their following useful
properties can be found, e.g. in \cite{A.S.T, Dra}. For every $x$ and $y$ in $X$ and
for every $\alpha = |\alpha|e^{i\theta}$, $\beta = |\beta|e^{i\omega}$ in $\mathbb{C}$, we have
\begin{enumerate}[]
\item (nd1) \label{nd1.perp}
$\rho_{_{-}}(x,y) = -\rho_{_{+}}(x,-y)$,
\item (nd2) \label{nd2.perp}
$\rho_{_{+}}(x,\alpha x + y) = {\rm Re}\,\alpha \,\|x\|^2 + \rho_{_{+}}(x,y)$,
\item (nd3)]\label{nd3.perp}
$\rho_{_{+}}(\alpha x,\beta y) = |\alpha \beta|\rho_{_{+}}(x,e^{i(\omega - \theta)}y)$,
\item (nd4) \label{nd4.perp}
$|\rho_{_{+}}(x,y)| \leq \|x\|\|y\|$,
\item (nd5) \label{nd5.perp}
$\rho_{_{+}}(x,y) =\|x\|\max\{{\rm Re}\,f(y): f\in J(x)\}$,
\item (nd6)  \label{nd6.perp}
$\rho_{_{+}}$ is continuous with respect to the second variable.
\end{enumerate}
Norm derivatives have applications
in studying the geometry of normed spaces. We refer the reader to
\cite{A.S.T, Amir, C.W.2, D.A.J, Dra, Ku.Sa, W.Z.L, Z.D}
and references therein for some of the prominent works on the context.

In recent years, several new mappings, as combinations of $\rho_{_{+}}$ and $\rho_{_{-}}$,
have been introduced.
In \cite{Milicic-1987}, Mili\v{c}i\'c introduced the mapping $\rho$ by
\begin{align}\label{eqn:dfn-of-rho}
\rho(x,y) = \frac{\rho_{_{+}}(x,y) + \rho_{_{-}}(x,y)}{2}.
\end{align}
For $\lambda\in[0,1]$, in \cite{Z.M}, the convex combination
$\rho_{_{\lambda}}$ of $\rho_{_{+}}$ and $\rho_{_{-}}$ is defined by
\begin{align*}
\rho_{_{\lambda}}(x,y)=\lambda\rho_-(x,y)+(1-\lambda)\rho_+(x,y).
\end{align*}
Later, in \cite{D.A.J}, for $\lambda\in[0,1]$ and $\upsilon=\frac{1}{2k-1}$ with $k\in\mathbb{N}$,
the mapping $\rho_{_{\lambda}}^{\upsilon}$
as a generalization of $\rho_\lambda$ is introduced as follows
\begin{align*}
\rho_{_{\lambda}}^{\upsilon}(x,y)
=\lambda\rho^{\upsilon}_{_{-}}(x,y)\rho^{1-\upsilon}_{_{+}}
(x,y)+(1-\lambda)\rho^{\upsilon}_{_{+}}(x,y)\rho^{1-\upsilon}_{_{-}}(x,y).
\end{align*}

Obviously, depending on the problem under consideration, some of these mappings may
appear more convenient and useful than the others
(see \cite{A.S.T, C.L.2015, Ch-Sik-Wo-2020, C.W.1, D.A.J,
Milicic-1987, M.Z.D, Sain 2021, Wo.2016, Wo.LAA.2019, Z.M}).

Recall that, in an inner product space, two elements are orthogonal if their inner product is zero.
The concept of orthogonality in an arbitrary normed space may be introduced
in various ways (e.g. see \cite{Alonso.Martini.Wu, A.S.T, Amir}).
In a normed space $X$ with a semi inner product $[\cdot, \cdot]$ a semi-orthogonality
of elements $x$ and $y$ is naturally defined by
\begin{align*}
x\perp_{s} y: \Leftrightarrow [y, x] = 0.
\end{align*}
Let $\diamondsuit\in\{\rho_{_{+}}, \rho_{_{-}}, \rho, \rho_{_{\lambda}}, \rho_{_{\lambda}}^{\upsilon}\}$.
Given $x, y \in X$, $x$ is said to be $\diamondsuit$-orthogonal to $y$, written as $x\perp_{\diamondsuit} y$, if $\diamondsuit(x, y) = 0$.
Another concept of orthogonality is the Birkhoff--James orthogonality (see \cite{B, J.1}).
It is said that $x$ is orthogonal to $y$ in the Birkhoff--James sense, in short,
$x\perp_{B} y$, if $\|x\| \leq \|x+\xi y\|$ for all $\xi \in\mathbb{C}$.
James in \cite[Theorem~2.1]{J.2} proved that if $x\in X\setminus\{0\}$
and $y\in X$, then $x\perp_{B} y$ if and only if there exists $f\in J(x)$ such that $f(y) = 0$.

The paper is organized as follows.
In Section 2, based on the right norm derivative $\rho_{_{+}}$,
we define a mapping $\rho_{_{\infty}}:X\times X\rightarrow \C$ by
\begin{equation*}
\rho_{_{\infty}}(x,y) = \frac1\pi\int_0^{2\pi}e^{i\theta}\rho_{_{+}}(x,e^{i\theta}y)d\theta,
\end{equation*}
and investigate its basic properties.
In particular, for arbitrary elements $x$ and $y$ of $X$ we show that
\begin{equation*}
|\rho_{_{\infty}}(x,y) |\leq \big(1+2\mathcal{R}(X^*)\big)\|x\|\|y\|.
\end{equation*}
In Section 3 we define an orthogonality relation in $X$ based on the mapping $\rho_{_{\infty}}$ as follows
\begin{align*}
x\perp_{\rho_{_{\infty}}} y: \Leftrightarrow \rho_{_{\infty}}(x,y) = 0 \quad (x, y \in X).
\end{align*}

For a given semi-inner product $[\cdot,\cdot]$ on $X$,
we prove that the condition $\perp_{\rho_{_{\infty}}} = \perp_{s}$ is equivalent to
$\rho_{_{\infty}}(x,y)= \overline{[y, x]}$, for all $x, y\in X$.
Further, we show that $\perp_{\rho_{_{\infty}}}\subseteq \perp_{B}$ if and only if the Cauchy--Schwarz inequality
$\big|\rho_{_{\infty}}(x,y)\big|\leq\|x\|\|y\|$ holds for all $x, y\in X$.
In Section 4, a characterization of complex inner product spaces is presented
in terms of $\rho_{_{\infty}}$; it is proved that $X$ is an inner product
space if and only if $\rho_\infty(x,y)=\rho_\infty(y,x)$, for all $x,y\in X$.
In the last section we consider a class of linear mappings preserving the $\rho_{_{\infty}}$-orthogonality.
More precisely, we show that, if a bounded linear mapping $T$ between complex normed spaces $X$ and $Y$ satisfies the following property
\begin{equation*}
\forall_{x, y \in X}\,\, x\perp_{\rho_{_{\infty}}} y \Longrightarrow Tx\perp_{\rho_{_{\infty}}} Ty,
\end{equation*}
then $\rho_{_{\infty}}(Tx,Ty) = \|T\|^2 \rho_{_{\infty}}(x,y)$ for all $x, y\in X$.
\section{A functional based on the norm derivatives}
Let $X$ be a normed space and consider the mapping $\rho$ defined in \eqref{eqn:dfn-of-rho}.
By (nd1), for every $x, y \in X$ we have
\begin{align}\label{eqn:rho-2}
\rho(x,y)
= \frac{\rho_{_{+}}(x,y)-\rho_{_{+}}(x,-y)}{2}
= \frac12\sum_{k=1}^2 c_k \rho_{_{+}}(x,c_ky),
\end{align}
where $c_1=1$ and $c_2=-1$ are the 2nd roots of unity.
If $X$ is a real inner product space, then $\rho(x,y)=\ip{x,y}$. However, in case $X$
is a complex inner product space, we have $\rho(x,y) = \re\ip{x,y}$.
The reason for this adverse situation is that $c_1^2+c_2^2=2$.
Something magical happens if we consider $n$th roots of unity, for $n>2$.
\begin{lem}\label{lem:sum=0}
Let $n>2$, and  suppose that $c_1, c_2,\dotsc, c_n$ are the $n$th roots of unity.
Then $\sum_{k=1}^n c_k^2 =0$.
\end{lem}
\begin{proof}
Let $c=e^{2\pi i/n}$. Then $c,c^2,\dotsc,c^n$ are all of the $n$th roots of unity.
In the equality $$(1-z)(1+z+z^2+\dotsb+z^{n-1}) = 1-z^n,$$ let $z=c^2$.
Since $n>2$, we have $c^2\neq 1$ and thus
\begin{equation*}
1+c^2+c^4+\dotsb+c^{2n-2} = 0,
\end{equation*}
meaning that $\sum_{k=1}^n c_k^2 =0$.
\end{proof}
Inspired by \eqref{eqn:rho-2}, we present the following definition in complex normed spaces.
For an overview of similarities and differences between real and complex normed spaces see \cite{M.M.P.S}.
\begin{dfn}
Suppose that $X$ is a complex normed space. Let $n>2$ and $c_1,c_2,\dotsc,c_n$
be the $n$th roots of unity. Define, for $x,y\in X$,
\begin{equation}\label{eqn:dfn:rho-n}
\rho_{_{n}}(x, y) := \frac2n\sum_{k=1}^n c_k \rho_{_{+}}(x, c_k y).
\end{equation}
\end{dfn}
In the following, some basic properties of the mapping $\rho_{_{n}}$ are established.
\begin{prop}\label{prop:properties-of-rho-n}
Let $(X,\|\!\cdot\!\|)$ be a complex normed space and $n>2$. Then
\begin{enumerate}[\upshape(i)]
\item $\rho_{_{n}}(x,x) = \|x\|^2$, for all $x\in X$,
\item $|\rho_{_{n}}(x,y)| \leq 2\|x\| \|y\|$, for all $x\in X$,
\item if the norm of $X$ comes from an inner product $\ip{\cdot,\cdot}$, then
$\rho_{_{n}}(x,y) = \ip{x,y}$ for all $x,y\in X$.
\end{enumerate}
\end{prop}
\begin{proof}
(i) Let $x\in X$. By (nd1) and Lemma \ref{lem:sum=0}, we have
\begin{align*}
\rho_{_{n}}(x,x)
& = \frac2n \sum_{k=1}^n c_k\rho_{_{+}}(x, c_k x)
      = \frac2n \sum_{k=1}^n c_k \re (c_k) \|x\|^2 \\
    & = \frac{\|x\|^2}{n} \sum_{k=1}^n c_k  (c_k + \bar c_k)
      = \frac{\|x\|^2}{n} \sum_{k=1}^n (c_k^2 + |c_k|^2) \\
    & = \frac{\|x\|^2}{n} \Bigl(\sum_{k=1}^n c_k^2 + \sum_{k=1}^n |c_k|^2 \Bigr)
      = \frac{\|x\|^2}{n} (0+n)
      = \|x\|^2.
\end{align*}

(ii) Let $x, y\in X$. By (nd4), we get
\begin{align*}
  |\rho_{_{n}}(x,y)|
    & = \frac2n \Bigl|\sum_{k=1}^n c_k\rho_{_{+}}(x, c_k y)\Bigr|
    \leq \frac2n \sum_{k=1}^n |c_k|\Bigl|\rho_{_{+}}(x, c_k y)\Bigr|
    \\& \leq \frac2n \sum_{k=1}^n |c_k|\|x\| \|c_k y\|
    = \frac2n \sum_{k=1}^n |c_k| \|x\| \|y\| = 2\|x\| \|y\|.
\end{align*}

(iii) Let the norm of $X$ come from an inner product $\ip{\cdot,\cdot}$.
Then, for every $x, y\in X$, we have
  \begin{align*}
    \rho_{_{n}}(x,y)
       & = \frac2n \sum_{k=1}^n c_k \rho_{_{+}}(x,c_k y)
         = \frac2n \sum_{k=1}^n c_k \re \ip{x,c_k y} \\
       & = \frac1{n} \sum_{k=1}^n c_k \bigl(\ip{x,c_k y} + \overline{\ip{x,c_k y}}\bigr)\\
       &  = \frac1{n} \sum_{k=1}^n c_k \bigl(\bar c_k\ip{x, y} + c_k \overline{\ip{x, y}}\bigr) \\
       &  = \frac{\ip{x, y}}{n} \sum_{k=1}^n |c_k|^2
             + \frac{\overline{\ip{x, y}}}{n} \sum_{k=1}^n c_k^2
         = \ip{x,y}.
         \qedhere
  \end{align*}
\end{proof}
To proceed, we recall the following fact for the class of continuous functions.
If $f:[0,1]\to\C$ is a continuous function, then
\begin{equation}
  \lim_{n\to \infty} \frac 1n\sum_{k=1}^n f\Bigl(\frac kn \Bigr)
    = \int_0^1 f(t)dt.
\end{equation}
Replacing $c_k$ with $e^{2k\pi i/n}$, we can rephrase \eqref{eqn:dfn:rho-n}
as follows
\begin{equation}\label{eqn:dfn:rho-n-2}
  \rho_{_{n}}(x,y) = \frac2n\sum_{k=1}^n e^{2k\pi i/n} \rho_{_{+}}(x,e^{2k\pi i/n} y).
\end{equation}
Now, letting $n\to \infty$, in \eqref{eqn:dfn:rho-n-2},
\begin{equation*}
  \lim_{n\to\infty }\frac2n\sum_{k=1}^n e^{2k\pi i/n} \rho_{_{+}}(x,e^{2k\pi i/n} y)
   = 2\int_0^1e^{i 2\pi t}\rho_{_{+}}(x, e^{i2\pi t}y) dt.
\end{equation*}
Applying substitution $\theta=2\pi t$, we get
\begin{equation}
   \lim_{n\to\infty}\rho_{_{n}}(x,y) = \frac1{\pi} \int_0^{2\pi}e^{i\theta}\rho_{_{+}}(x, e^{i\theta}y) d\theta.
\end{equation}
This motivates us to present the following definition.
\begin{dfn}\label{dfn:rho-infty}
  Let $X$ be a complex normed space. Define, for $x,y\in X$,
  \begin{equation*}
    \rho_{_{\infty}}(x,y) = \lim_{n\to\infty}\rho_{_{n}}(x,y)
        = \frac1{\pi}\int_{0}^{2\pi}e^{i\theta}\rho_{_{+}}(x,e^{i\theta}y)d\theta.
  \end{equation*}
\end{dfn}
\begin{rem}
If the norm of $X$ comes from an inner product $\ip{\cdot,\cdot}$, then
by Proposition \ref{prop:properties-of-rho-n}(iii), for all $x,y\in X$, we have
\begin{align*}
\rho_{_{\infty}}(x,y)=\lim_{n\to\infty}\rho_{_{n}}(x,y)=\ip{x,y}.
\end{align*}
\end{rem}

Before further investigation, we provide our discussion with an example.
\begin{example}\label{exa:rho-infty-in-ell-1}
  Let $X=\ell^1$ be the space of all summable complex sequences
  with ${\|\!\cdot\!\|}_1$. For two sequences $x=(x_k)$ and $y=(y_k)$ in $\ell^1$, and
  for every $\theta \in \R$, we have
  \begin{align}
    \rho_{_{+}}(x,e^{i\theta}y)
       & = {\|x\|}_1  \lim_{t\to0+}\frac{{\|x+ t e^{i\theta}y\|}_1-{\|x\|}_1}{t} \notag\\
       & = {\|x\|}_1  \lim_{t\to0+}\sum_{k=1}^\infty\frac{|x_k+ t e^{i\theta}y_k|-|x_k|}{t} \notag\\
       & = {\|x\|}_1 \lim_{t\to0} \Biggl(\sum_{x_k\neq0} |x_k|  \frac{|1+t e^{i\theta}\frac{y_k}{x_k}|-1}t
                         + \sum_{x_k=0} |e^{i\theta}y_k|  \Biggr) \notag\\
       & = {\|x\|}_1 \lim_{t\to0} \Biggl(\sum_{x_k\neq0} |x_k|
                       \frac{2t\re \bigl(e^{i\theta}\frac{y_k}{x_k}\bigr) +t^2 \bigl|\frac{y_k}{x_k}\bigr|^2}
                        {t\bigl(\bigl|1+t e^{i\theta}\frac{y_k}{x_k}\bigr|+1\bigr)}
                         + \sum_{x_k=0} |y_k|  \Biggr) \notag\\
       & = {\|x\|}_1 \Biggl(\sum_{x_k\neq0} |x_k| \re\Bigl(e^{i\theta}\frac{y_k}{x_k} \Bigr)
           + \sum_{x_k=0} |y_k|  \Biggr). \label{eqn:rho+in-ell-1}
  \end{align}
Therefore,
\begin{align*}
  \rho_{_{\infty}}(x,y)
    & = \frac1{\pi} \int_0^{2\pi} e^{i\theta}\rho_{_{+}}(x,e^{i\theta}y) d\theta \\
    & = \frac{{\|x\|}_1}{\pi} \Biggl(
       \sum_{x_k\neq0} |x_k| \int_0^{2\pi} e^{i\theta}\re\Bigl(e^{i\theta}\frac{y_k}{x_k}\Bigr) d\theta
           + \sum_{x_k=0} |y_k| \int_0^{2\pi} e^{i\theta } d\theta\Biggr)\\
    & = \frac{{\|x\|}_1}{\pi} \Biggl(
       \sum_{x_k\neq0} |x_k| \int_0^{2\pi} \frac{e^{i\theta}}2\Bigl(e^{i\theta}\frac{y_k}{x_k}+e^{-i\theta}\frac{\bar y_k}{\bar x_k}\Bigr) d\theta
           + 0\Biggr)\\
    & = \frac{{\|x\|}_1}{\pi} \sum_{x_k\neq0} \frac{|x_k|}2 \Bigl(\frac{y_k}{x_k}
         \int_0^{2\pi} e^{2i\theta}d\theta + \frac{\bar y_k}{\bar x_k}\int_0^{2\pi}d\theta\Bigr)\\
    & = \frac{{\|x\|}_1}{\pi} \sum_{x_k\neq0} \frac{|x_k|}2 \Bigl(0 + 2\pi \frac{\bar y_k}{\bar x_k}\Bigr)\\
    & = {\|x\|}_1 \sum_{x_k\neq0} \frac{x_k \bar y_k}{|x_k|}.
\end{align*}
Hence,
\begin{equation}\label{eqn:rho-infty-in-ell-1}
 \rho_{_{\infty}}(x,y) = {\|x\|}_1 \sum_{x_k\neq0} \frac{x_k \bar y_k}{|x_k|}
 \quad (x,y\in \ell^1).
\end{equation}
\end{example}
In the following, basic properties of $\rho_{_{\infty}}$ are investigated.
Since the integrand functions in the discussion are $2\pi$-periodic,
we will utilize the following equality in our discussion several times;
  \begin{equation}\label{eqn:2pi-periodic}
    \int_{\phi}^{2\pi+\phi} u(\theta)d\theta
     = \int_0^{2\pi} u(\theta)d\theta \quad\big(\text{$u(\cdot)$ is $2\pi$-periodic, $\phi\in\R$}\big).
  \end{equation}
\begin{prop}\label{prop:rho-infty(ax,ax+y)}
  Let $(X, \|\!\cdot\!\|)$ be a complex normed space, $x,y\in X$, and $\alpha,\beta\in\C$. Then
  \begin{enumerate}[\upshape(i)]
    \item \label{item:rho-infty(alpha x,y)}
      $\rho_{_{\infty}}(\alpha x,\beta y)= \alpha \bar{\beta} \rho_{_{\infty}}(x,y)$.

    \item \label{item:rho-infty(x,alpha x+y)}
       $\rho_{_{\infty}}(x,\alpha x + y)
       = \rho_{_{\infty}}(x,\alpha x) + \rho_{_{\infty}}(x,y)
       = \bar\alpha \|x\|^2 + \rho_{_{\infty}}(x,y)$.
  \end{enumerate}
\end{prop}
\begin{proof}
  \eqref{item:rho-infty(alpha x,y)} Let $\alpha = |\alpha| e^{i\phi}$ and $\beta = |\beta| e^{i\psi}$,
  for some $\phi,\psi\in [0,2\pi)$. Then, by \eqref{eqn:2pi-periodic},
  \begin{align*}
    \rho_{_{\infty}}(\alpha x,\beta y)
      & = \frac1{\pi}\int_{0}^{2\pi}e^{i\theta} \rho_{_{+}}(|\alpha|e^{i\phi}x , |\beta|e^{i(\theta+\psi)}y)d\theta \\
      & = \frac{|\alpha\beta|}{\pi}\int_{0}^{2\pi}e^{i\theta} \rho_{_{+}}(x,e^{i(\theta+\psi-\phi)}y)d\theta \\
      & = \frac{|\alpha\beta|}{\pi}\int_{\psi-\phi}^{2\pi+\psi-\phi}e^{i(\xi+\phi-\psi)} \rho_{_{+}}(x,e^{i\xi}y)d \xi \\
      & = \frac{|\alpha\beta|e^{i(\phi-\psi)}}{\pi}\int_0^{2\pi}e^{i\xi} \rho_{_{+}}(x,e^{i\xi}y)d\xi \\
      & = \alpha \bar\beta\rho_{_{\infty}}(x,y).
  \end{align*}

  \eqref{item:rho-infty(x,alpha x+y)}
   First, note that
  \begin{align*}
    \int_0^{2\pi} e^{i\theta} \re(e^{i\theta}\alpha)d\theta
       &  = \frac12 \int_0^{2\pi} e^{i\theta} (e^{i\theta}\alpha+e^{-i\theta} \bar\alpha) d\theta \\
       &  = \frac12 \int_0^{2\pi} (e^{2\theta i} \alpha + \bar\alpha) d\theta \\
       &  = \frac\alpha 2 \int_0^{2\pi} e^{2\theta i} d\theta + \frac{\bar\alpha}2 \int_0^{2\pi} d\theta
        = 0 + \pi \bar \alpha = \pi \bar\alpha.
  \end{align*}
Now, write
  \begin{align*}
    \rho_{_{\infty}}(x, \alpha x + y)
      & = \frac1{\pi}\int_{0}^{2\pi}e^{i\theta} \rho_{_{+}}\bigl(x,e^{i\theta}(\alpha x + y)\bigr)d\theta \\
      & = \frac1{\pi}\int_{0}^{2\pi}e^{i\theta} \rho_{_{+}}(x,e^{i\theta}\alpha x + e^{i\theta} y)d\theta\\
      & =  \frac1{\pi}\int_{0}^{2\pi}e^{i\theta} \bigl( \re(e^{i\theta}\alpha)\|x\|^2+\rho_{_{+}}(x,e^{i\theta}y) \bigr)d\theta\\
      & =  \frac{\|x\|^2}\pi \int_{0}^{2\pi}e^{i\theta} \re(e^{i\theta}\alpha)d\theta
           + \frac1{\pi}\int_{0}^{2\pi}e^{i\theta} \rho_{_{+}}(x,e^{i\theta}y) d\theta\\
      & = \bar\alpha \|x\|^2+ \rho_{_{\infty}}(x,y).
    \end{align*}
\end{proof}
Here we present one of the main results of this paper.
\begin{thm}\label{thm:Cauchy-Sw-inequality}
 Let $(X, \|\!\cdot\!\|)$ be a complex normed space, and $x,y\in X$. Then
 \begin{equation}\label{eqn:CS-inequality-like}
   |\rho_{_{\infty}}(x,y)|\leq \big(1+2\mathcal{R}(X^*)\big)\|x\|\, \|y\|.
 \end{equation}
\end{thm}

\begin{proof}
  It is clear that there is a number $\gamma\in\mathbb{C}$, with $|\gamma|=1$, such
  that $|\rho_{_{\infty}}(x,y)|= \gamma \rho_{_{\infty}}(x,y)$.
Applying Proposition \ref{prop:rho-infty(ax,ax+y)}(i), we get
  $|\rho_{_{\infty}}(x,y)|= \rho_{_{\infty}}(x,\overline{\gamma}y)$.
  Put $z=\overline{\gamma}y$. Then we have
  \begin{align*}
    \pi |\rho_{_{\infty}}(x,y)|
      & = \pi\rho_{_{\infty}} (x,z) =\int_{0}^{2\pi}e^{i\theta}\rho_{_{+}}(x,e^{i\theta}z)d\theta.
  \end{align*}
Since $e^{i\theta}=\cos \theta + i \sin \theta$, we obtain
  \begin{equation}\label{eqn:|rho(x,y)|}
     \pi |\rho_{_{\infty}}(x,y)|=\int_{0}^{2\pi}\cos \theta\rho_{_{+}}(x,e^{i\theta}z)d\theta +i \int_{0}^{2\pi}\sin \theta\rho_{_{+}}(x,e^{i\theta}z)d\theta.
  \end{equation}
Since the left side in \eqref{eqn:|rho(x,y)|} is real, the imaginary
part in the right side should be zero, and thus
  \begin{equation}\label{equality-pi-rho-int-xz}
    \pi |\rho_{_{\infty}} (x,y)|=\int_{0}^{2\pi}\cos \theta\rho_{_{+}}(x,e^{i\theta}z)d\theta.
  \end{equation}
It follows from (nd5) that for all $\theta\in[0,2\pi)$ there exists
a functional $f_{\theta}\in J(x)$ such that
\begin{equation}\label{equality-rho-re-cos-sin}
 \rho_{_{+}}(x,e^{i\theta}z)
     =\|x\| \re f_{\theta}(e^{i\theta}z).
\end{equation}
Fix an arbitrary $f\in J(x)$. It is easy to check that
${\rm conv}\{f_{\theta},f\}\subseteq S_{X^*}$. Hence $\|f_{\theta}-f\|\leq\mathcal{R}(X^*)$.
Moreover, the equalities (\ref{equality-pi-rho-int-xz}) and (\ref{equality-rho-re-cos-sin})
become
\begin{align}
    \pi |\rho_{_{\infty}} (x,y)|
      & = \|x\| \int_{0}^{2\pi}\cos \theta\re  f_{\theta}(e^{i\theta}z)d\theta\nonumber\\
      & =\|x\| \left(\int_{0}^{2\pi}\cos \theta\re (f_{\theta}-f)(e^{i\theta}z)d\theta+\int_{0}^{2\pi}\cos \theta\re f(e^{i\theta}z)d\theta\right)\label{second-inetgral-re}.
 \end{align}
Next we have
\begin{align*}
\int_{0}^{2\pi}\cos \theta\re f(e^{i\theta}z)d\theta
      & = \int_{0}^{2\pi}\cos\theta \re f\big((\cos \theta+i \sin \theta) z\big)d\theta\\
      & = \int_{0}^{2\pi}\left(\cos^2\theta\re f\big(z\big)+\cos\theta\sin\theta \re f\big(i z\big)\right)d\theta\\
      & = \re f\big(z\big) \int_{0}^{2\pi}\cos^2\theta d\theta+\re f\big(i z\big)
      \int_{0}^{2\pi}\frac{1}{2}\sin2\theta d\theta.
 \end{align*}

Since $\int_{0}^{2\pi}\cos^2\theta d\theta=\pi$ and
$\int_{0}^{2\pi}\frac{1}{2}\sin2\theta=0$, it follows
from the above equalities that
 \begin{align}\label{integral-re-eq-zero}
\int_{0}^{2\pi}\cos \theta\re f(e^{i\theta}z)d\theta = \pi\re f\big(z\big).
\end{align}

\noindent
Combining (\ref{second-inetgral-re}) with (\ref{integral-re-eq-zero}) yields
 \begin{align*}
    \pi |\rho_{_{\infty}} (x,y)|
      & =\|x\| \left(\int_{0}^{2\pi}\cos \theta \re (f_{\theta}-f)(e^{i\theta}z)d\theta+\pi \re f\big(z\big) \right)\\
      & \leq \|x\| \left(\int_{0}^{2\pi}\|f_{\theta}-f\| \|z\|d\theta + \pi\|z\|\right)\\
      & \leq \|x\| \|z\| \left(2\pi\|f_{\theta}-f\|+\pi\right) \\
      & \leq \|x\|\|z\|\left(2\pi\mathcal{R}(X^*)+\pi\right).
 \end{align*}
Since $\|y\|=\|z\|$, it follows from the above estimation that
\begin{equation*}
  |\rho_{_{\infty}} (x,y)| \leq (1+2\mathcal{R}(X^*))\|x\| \|y\|.
\end{equation*}
\end{proof}
As an immediate consequence of the above result, we obtain the following.
\begin{cor}\label{cor:Cauchy-Sw-inequality-smooth-sp}
Let $(X, \|\!\cdot\!\|)$ be a complex reflexive smooth space. Then
    \begin{equation*}
      |\rho_{_{\infty}}(x,y)|\leq \|x\|\|y\| \quad(x,y\in X).
    \end{equation*}
\end{cor}
\begin{proof}
Since $X$ is a reflexive smooth space, $X^*$ is rotund.
Therefore, $\mathcal{R}(X^*)=0$. Now, utilizing Theorem~\ref{thm:Cauchy-Sw-inequality}, we deduce the desired result.
\end{proof}
\begin{rem}
  As the above corollary shows, in the inequality \eqref{eqn:CS-inequality-like},
  the constant $R=1+2\mathcal{R}(X^*)$ is the best for reflexive smooth spaces.
  If we continue \eqref{equality-pi-rho-int-xz} in a slightly different direction,
  we reach the following
  \begin{align*}
    \pi |\rho_{_{\infty}} (x,y)|
      & = \int_{0}^{2\pi} \cos \theta \rho_{_{+}}(x,e^{i\theta}z)d\theta\\
      & = \Bigl|\int_{0}^{2\pi} \cos \theta \rho_{_{+}}(x,e^{i\theta}z)d\theta\Bigr| \\
      & \leq \int_{0}^{2\pi} |\cos \theta \rho_{_{+}}(x,e^{i\theta}z)|d\theta \\
      & \leq \Bigl(\int_{0}^{2\pi}|\cos \theta| d\theta\Bigr) \|x\| \|z\|
        = 4 \|x\| \|y\|.
  \end{align*}

  Therefore, $|\rho_\infty (x,y)| \leq \frac4\pi \|x\| \|y\|$. This inequality
  is a good alternative to \eqref{eqn:CS-inequality-like}, in some normed spaces.
  For instance, take $X=\ell^1$, as in Example \ref{exa:rho-infty-in-ell-1}.
  Then $X^*=\ell^\infty$ and $\mathcal{R}(X^*)=2$.
  In this case, $\frac4\pi < 1+2\mathcal{R}(X^*)$.

  Indeed, the Cauchy--Schwartz inequality holds in $\ell^1$; for $x=(x_k)$ and $y=(y_k)$ in $\ell^1$,
  by \eqref{eqn:rho-infty-in-ell-1}, we have
  \begin{equation}\label{eqn:CS-inequality-in-ell-1}
    |\rho_{_{\infty}}(x,y)|
    = {\|x\|}_1 \Bigl|\sum_{x_k\neq0} \frac{x_k \bar y_k}{|x_k|}\Bigr|
    \leq {\|x\|}_1 \sum_{x_k\neq0} |\bar y_k| \leq {\|x\|}_1 {\|y\|}_1.
  \end{equation}

  One may conjecture that the Cauchy--Schwartz inequality holds in every complex normed space;
  we do not currently have strong evidence supporting it.
  See also Theorem \ref{thm:CS-inequality-iff} below.
\end{rem}
We finish this section with the following result.
\begin{thm}\label{prop:two-norms-are-equiv-iff}
Let $X$ be a complex space endowed with two norms ${\|\!\cdot\!\|}_1$ and ${\|\!\cdot\!\|}_2$.
Then the following conditions are equivalent.
  \begin{enumerate}[\upshape(i)]
    \item \label{item:norms-are-equivalent}
       The norms ${\|\!\cdot\!\|}_1$ and ${\|\!\cdot\!\|}_2$ are equivalent.
    \item \label{item:exists-constant-c}
       There exists a positive constant $c$ such that, for all $x,y\in X$,
    \begin{equation*}
      |\rho_{_{\infty,1}}(x,y)-\rho_{_{{\infty,2}}}(x,y)|
       \leq c \min\set{{\|x\|}_1{\|y\|}_1, {\|x\|}_2{\|y\|}_2},
    \end{equation*}
    where $\rho_{_{{\infty,k}}}$ is the functional $\rho_{_{\infty}}$
    with respect to ${\|\!\cdot\!\|}_k$, for $k=1,2$.
  \end{enumerate}
\end{thm}
\begin{proof}
In the following, we let $R=1+2\mathcal{R}(X^*)$.

Suppose that \eqref{item:norms-are-equivalent} holds. Then, there exist positive numbers
$m, M$ such that
$m{\|z\|}_1\leq {\|z\|}_2  \leq M{\|z\|}_1$ for $z\in X$.
Theorem \ref{thm:Cauchy-Sw-inequality} implies that, for all $x,y\in X$,
    \begin{equation*}
      |\rho_{_{{\infty,1}}}(x, y)| \leq R{\|x\|}_1 {\|y\|}_1,
    \end{equation*}
and
\begin{equation*}
      |\rho_{_{{\infty,2}}}(x, y)| \leq R{\|x\|}_2 {\|y\|}_2.
    \end{equation*}
Hence,
\begin{align*}
|\rho_{_{{\infty,1}}}(x, y) - \rho_{_{{\infty,2}}}(x, y)|\leq R\bigl({\|x\|}_1{\|y\|}_1 + {\|x\|}_2{\|y\|}_2\bigr),
\end{align*}
and so
\begin{align*}
|\rho_{_{{\infty,1}}}(x, y) - \rho_{_{{\infty,2}}}(x, y)|\leq R(1+M^2){\|x\|}_1{\|y\|}_1.
\end{align*}
Similarly,
    \begin{equation*}
      |\rho_{_{{\infty,1}}}(x, y) - \rho_{_{{\infty,2}}}(x, y)|
        \leq R\bigl(1+\frac{1}{m^2}\bigr){\|x\|}_2 {\|y\|}_2.
    \end{equation*}
Put $c= R\bigl(1+\max\set{M^2,\frac{1}{m^2}}\bigr)$. Therefore, we obtain
    \begin{equation*}
      |\rho_{_{{\infty,1}}}(x, y) - \rho_{_{{\infty,2}}}(x, y)|
        \leq c\min\set{{\|x\|}_1{\|y\|}_1, {\|x\|}_2{\|y\|}_2}.
    \end{equation*}

Conversely, assume that \eqref{item:exists-constant-c} holds for some $c>0$.
Then, for every $x\in X$, we have
\begin{equation*}
      |\rho_{_{{\infty,1}}}(x, x) - \rho_{_{{\infty,2}}}(x, x)|
        \leq c\min\set{{\|x\|}^2_1, {\|x\|}^2_2}.
    \end{equation*}
Therefore,
    \begin{equation*}
      \bigl|{\|x\|}^2_1 - {\|x\|}^2_2\bigr| \leq c{\|x\|}^2_1, \quad
      \bigl|{\|x\|}^2_1 - {\|x\|}^2_2\bigr| \leq c{\|x\|}^2_2,
    \end{equation*}
and thus,
    \begin{equation*}
      {\|x\|}_2 \leq \sqrt{1+c} {\|x\|}_1, \quad
      {\|x\|}_1 \leq \sqrt{1+c} {\|x\|}_2.
    \end{equation*}
Hence,
    \begin{equation*}
      \frac1{\sqrt{1+c}}{\|x\|}_1\leq {\|x\|}_2  \leq \sqrt{1+c}{\|x\|}_1 \quad (x\in X),
    \end{equation*}
meaning that the two norms are equivalent.
\end{proof}
\section{$\rho_{_{\infty}}$-orthogonality}
We start this section with an orthogonality relation in $X$ based on the mapping $\rho_{_{\infty}}$ as follows.
\begin{dfn}
Let $(X, \|\cdot\|)$ be a complex normed space. For elements $x, y\in X$, define $x\perp_{\rho_{_{\infty}}} y$
if $\rho_{_{\infty}}(x,y)=0$.
\end{dfn}
We will need the following lemma in the sequel.

\begin{lem}\label{prop:exists-a-such-that-rho-infty(x,ax+y)=0}
  For every $x,y\in X$, there is $\alpha\in \C$ such that
  $x \perp_{\rho_{_{\infty}}} (\alpha x + y)$.
\end{lem}
\begin{proof}
  Assume, without loss of generality, that $x\neq 0$.
  Put $\alpha = -\frac{\overline{\rho_{_{\infty}}(x,y)}}{\|x\|^2}$.
  By Proposition \ref{prop:rho-infty(ax,ax+y)}~(\ref{item:rho-infty(x,alpha x+y)}),
  we have $\rho_{_{\infty}}(x,\alpha x+y)=0$. Thus $x \perp_{\rho_{_{\infty}}} (\alpha x + y)$.
\end{proof}
In a complex normed space equipped with various types of orthogonality, it is always
of some interest to consider relations between any two of the given orthogonalities, or to
characterize the classes of spaces in which one type of orthogonality implies the other.

\begin{thm}\label{thm:perp-infinity-and-perp-s}
 Let $(X, \|\!\cdot\!\|)$ be a complex normed space and let $[\cdot, \cdot ]$ be a semi-inner
 product in $X$. The following conditions are equivalent.
 \begin{enumerate}[\upshape(i)]
   \item \label{item:perp-infinity=perp-s}
     $\perp_{\rho_{_{\infty}}} = \perp_s$,
   \item \label{item:perp-infinity-subset-perp-s}
     $\perp_{\rho_{_{\infty}}} \subseteq \perp_s$,
     \item \label{item:perp-infinity-supset-perp-s}
     $\perp_{\rho_{_{\infty}}} \supseteq \perp_s$,
   \item \label{item:rho-infinity=ip-s}
     $\rho_{_{\infty}}(x,y)= \overline{[y,x]}$, for all $x,y\in X$.
 \end{enumerate}
\end{thm}
\begin{proof}
 The implications \eqref{item:rho-infinity=ip-s} $\Rightarrow$ \eqref{item:perp-infinity=perp-s}
 $\Rightarrow$ \eqref{item:perp-infinity-subset-perp-s} and
 \eqref{item:perp-infinity=perp-s}  $\Rightarrow$ \eqref{item:perp-infinity-supset-perp-s}  trivially hold.
We prove
 \eqref{item:perp-infinity-subset-perp-s} $\Rightarrow$ \eqref{item:rho-infinity=ip-s}.
Let $ x,y\in X$. Put $z:=-\frac{\overline{\rho_{_{\infty}}(x,y)}}{\|x\|^2}x+y$. Then,
Lemma \ref{prop:exists-a-such-that-rho-infty(x,ax+y)=0} yields $x\perp_{\rho_{_{\infty}}} z$.
By \eqref{item:perp-infinity-subset-perp-s}, we obtain $x\perp_{s} z$. Whence,
\begin{align*}
0 = [z,x]= -\frac{\overline{\rho_{_{\infty}}}(x,y)}{\|x\|^2}\|x\|^2 + [y,x],
\end{align*}
and thus $\rho_{_{\infty}}(x,y) = \overline{[y,x]}$.
The implication \eqref{item:perp-infinity-supset-perp-s} $\Rightarrow$ \eqref{item:rho-infinity=ip-s}
is proved similarly (using (sip1-3) and Proposition \ref{prop:rho-infty(ax,ax+y)}).
\end{proof}
The following example shows that the relations
$\perp_{\rho_{_{+}}}$ and $\perp_{\rho_{_{\infty}}}$ are not comparable in general.
\begin{example}
  Take $X=\ell^1$ as in Example \ref{exa:rho-infty-in-ell-1}. Let
  \begin{equation*}
   \left\{\!\!\!
    \begin{array}{ll}
      x=(1,0,0,\dotsc),\\
      y=(i,0,0,\dotsc),
    \end{array}
  \right. \qquad
  \left\{\!\!\!
    \begin{array}{ll}
      u=(1,1,0,0,\dotsc),\\
      v=(1,-1,0,0,\dotsc).
    \end{array}\right.
  \end{equation*}
Then, using \eqref{eqn:rho+in-ell-1}, we get
$\rho_{_{+}}(x,y) = 0$ and $\rho_{_{+}}(u,v) = 1$.
While, by \eqref{eqn:rho-infty-in-ell-1}, we have
$\rho_{_{\infty}}(x,y) = -i$ and $\rho_{_{\infty}}(u,v) = 0$.
We conclude that $\perp_{\rho_{_{+}}} \not\subset \perp_{\rho_{_{\infty}}}$
and $\perp_{\rho_{_{\infty}}} \not\subset \perp_{\rho_{_{+}}}$.
\end{example}
The above example motivates the next result.
\begin{prop}
Let $X$ be a complex normed space. Then the
inclusion $\perp_{\rho_{_{+}}}\subseteq\perp_{\rho_{_{\infty}}}$ does not hold.
\end{prop}
\begin{proof}
Suppose, for a contradiction, that the inclusion $\perp_{\rho_{_{+}}}\subseteq\perp_{\rho_{_{\infty}}}$
is true. In a similar way, as in the proof of Theorem \ref{thm:perp-infinity-and-perp-s}, we
obtain $\rho_{_{+}}(x,y)=\overline{\rho_{_{\infty}}(x,y)}$ for all $x,y\in X$.
Fix $x,y\in X$ such
that $\rho_{_{+}}(x,y)\neq 0$. By Proposition \ref{prop:rho-infty(ax,ax+y)}(i), we obtain
$\rho_{_{+}}(x,e^{i\theta}y) = e^{i\theta}\overline{\rho_{_{\infty}}(x,y)}$,
for all $\theta\in[0,2\pi)$. It is clear that
$e^{i\theta_o}\overline{\rho_{_{\infty}}(x, y)}\in \C\setminus\R$, for some $\theta_o\in[0,2\pi)$.
On the other hand, $\rho_{_{+}}(x,e^{i\theta_o}y)\in \R$ and we have our desired contradiction.
\end{proof}

The following example shows that, in general, $\perp_{B} \not\subset \perp_{\rho_{_{\infty}}}$.
\begin{example}\label{exa:+perp(B)-perp(rho-infity)}
  Take $X=\ell^1$, as in Example \ref{exa:rho-infty-in-ell-1}. Let
  \begin{equation*}
    x=(0,1,0,0,\dotsc) \quad \text{and} \quad y=(2,1,0,0,\dotsc).
  \end{equation*}
For every $\xi\in\C$ we have
\begin{align*}
{\|x+\xi y\|}_1 =2|\xi| + |1+\xi| \geq 1 = {\|x\|}_1.
\end{align*}
This shows that $x\perp_{B} y$. However, by \eqref{eqn:rho-infty-in-ell-1}, we have
$\rho_{_{\infty}}(x,y) =1$,
and thus $x\not\perp_{\rho_{_{\infty}}}y$.
\end{example}
It is natural to ask whether the inclusion $\perp_{\rho_{_{\infty}}}\subseteq \perp_{B}$ hold.
The next theorem gives a partial answer (it is motivated by Theorem~\ref{thm:Cauchy-Sw-inequality}).

\begin{thm}\label{thm:CS-inequality-iff}
Let $(X, \|\!\cdot\!\|)$ be a complex normed space. Then the following conditions are equivalent.
  \begin{enumerate}[\upshape(i)]
  \item \label{inequality-constant-one}
       $\big|\rho_{_{\infty}}(x,y)\big|\leq\|x\|\|y\|$ for all $x, y\in X$,
  \item \label{rho-inf-ort-subset-b-ort}
       $\perp_{\rho_{_{\infty}}}\subseteq \perp_{B}$.
  \end{enumerate}
\end{thm}

\begin{proof}
   To prove (\ref{inequality-constant-one})$\Rightarrow$(\ref{rho-inf-ort-subset-b-ort}), fix
   two arbitrary elements $x,y\in X$ with $x\perp_{\rho_{_{\infty}}}y$. Let $\xi\in\C$. Then,
   from Proposition \ref{prop:rho-infty(ax,ax+y)}~(\ref{inequality-constant-one})
   and (\ref{item:rho-infty(x,alpha x+y)}), it follows that
\begin{align*}
\|x\|^2 = |\rho_{_{\infty}}(x,x+\xi y)| \leq \|x\| \|x+\xi y\|.
\end{align*}
If $x\neq 0$, we get $\|x+\xi y\|\geq \|x\|$, which means that $x\perp_{B} y$.

Now, suppose that (\ref{rho-inf-ort-subset-b-ort}) holds. Fix $x,y\in X$. We may assume that $x\neq 0$.
By Lemma \ref{prop:exists-a-such-that-rho-infty(x,ax+y)=0} we have $x\perp_{\rho_\infty} -\frac{\overline{\rho_{_{\infty}}(x,y)}}{\|x\|^2}x+y$.
Applying (\ref{rho-inf-ort-subset-b-ort}) we get $x\perp_{B} -\frac{\overline{\rho_{_{\infty}}(x,y)}}{\|x\|^2}x+y$.
Now, the classical James' characterization of the Birkhoff--James orthogonality says that
there is a functional $f\in J(x)$ such that $f\left(-\frac{\overline{\rho_{_{\infty}}(x,y)}}{\|x\|^2}x+y\right)=0$.
Thus
\begin{align*}
f(y)=\frac{\overline{\rho_{_{\infty}}(x,y)}}{\|x\|^2}f(x) =\frac{\overline{\rho_{_{\infty}}(x,y)}}{\|x\|}.
\end{align*}
It follows that $\big|\rho_{_{\infty}}(x,y)\big|=\big|\overline{f(y)}\|x\|\big|\leq\|y\| \|x\|$, and we are done.
\end{proof}
\begin{rem}
Take $X=\ell^1$, as in Example \ref{exa:rho-infty-in-ell-1}.
According to \eqref{eqn:CS-inequality-in-ell-1}, the inequality
$\big|\rho_{_{\infty}}(x,y)\big|\leq\|x\|\|y\|$ holds, for all $x,y\in \ell^1$.
Therefore, in $\ell^1$, we have $\perp_{\rho_{_{\infty}}}\subset \perp_{B}$.
\end{rem}
\section{A characterization of inner product spaces}
The problem of finding necessary and sufficient conditions for a complex normed
space to be an inner product space has been investigated by many mathematicians,
see e.g. \cite{Amir} and the references therein.
In this section a characterization of complex inner product spaces is presented in terms of $\rho_{_{\infty}}$.

Let us quote a result from \cite{A.S.T}.
\begin{lem}[{\cite[Theorem 1.4.5]{A.S.T}}]\label{Theorem 1.4.5(Alsina)}
Let $X$ be a real normed space of dimension greater than or equal to $2$.
The space $X$ is an inner product space if and only if each two-dimensional
subspace of $X$ is an inner product space.
\end{lem}
If $(X, \|\!\cdot\!\|)$ is a normed space over the complex field which, as a space over $\R$,
has an inner product $\langle\cdot,\cdot \rangle_\R$, then
\begin{align*}
\langle x,y \rangle: = \langle x, y \rangle_\R + i \langle x, iy \rangle_\R
\end{align*}
is a complex inner product for $X$ as a vector space over $\C$. Note that
\begin{align*}
2\|x\|^2 = \|(1 + i)x\|^2 = 2\|x\|^2 + 2\langle x, ix \rangle_\R,
\end{align*}
so that $\langle x, ix \rangle_\R=0$, for all $x\in X$.
\begin{thm}\label{CIPS}
In a complex normed space $(X, \|\cdot\|)$, the following are equivalent.
  \begin{enumerate}[\upshape(i)]
    \item $X$ is a complex inner product space,
    \item $\rho_{_{\infty}}(x,y) = \rho_{_{\infty}}(y,x)$, for all $x,y\in X$.
  \end{enumerate}
\end{thm}

\begin{proof}
Obviously, (i)$\Rightarrow$(ii).

  Assume that $\rho_{_{\infty}}(x,y) = \rho_{_{\infty}}(y,x)$, for all $x,y\in X$.
  Consider $X$ as a real normed space, and let $E$ be any two-dimensional subspace of $X$.
  Define a mapping $\langle\cdot,\cdot\rangle_\R:E\times E \to \R$, by
  \begin{equation}\label{eqn:real-inner-product}
    \ip{x,y}_\R= \re \rho_{_{\infty}}(x,y)
     = \frac1{\pi} \int_0^{2\pi} \cos \theta \rho_{_{+}}(x,e^{i\theta}y) d\theta, \quad (x,y\in X).
  \end{equation}

  We show that $\langle\cdot,\cdot\rangle_\R$ is an inner product in $E$.
  It is easily seen that the function is symmetric, homogeneous and nonnegative.
  We only need to verify the additivity of $\langle\cdot,\cdot\rangle_\R$ in each variable.
  By the symmetry, however, it is enough to show the additivity with respect to the second variable.
  Take $x, y, z \in E$ and consider two cases. Assume first that $x, y$ are linearly dependent,
  so that $y = \lambda x$, for some $\lambda \in\R$. Then
  \begin{align*}
    \ip{x,y+z}_\R
      & = \ip{x,\lambda x +z}_\R \\
      & = \re \rho_{_{\infty}}(x, \lambda x+z) \\
      & = \re (\rho_{_{\infty}}(x,\lambda x) + \rho_{_{\infty}}(x,z)) \\
      & = \re \rho_{_{\infty}}(x,y) + \re \rho_{_{\infty}}(x,z) \\
      & = \ip{x,y}_\R + \ip{x,z}_\R.
  \end{align*}

  Now let $x, y$ be linearly independent. Then $z = \alpha x+ \beta y$, for some $\alpha,\beta \in \R$.
  Then by Proposition \ref{prop:rho-infty(ax,ax+y)}~(\ref{item:rho-infty(x,alpha x+y)}) we have
  \begin{align*}
    \ip{x,y+z}_\R
      & =\ip{x, \alpha x + (1+\beta) y}_\R \\
      & = \ip{x,\alpha x}_\R + \ip{x,(1+\beta) y}_\R \\
      & = \ip{x,\alpha x}_\R + (1+\beta)\ip{x,y}_\R \\
      & = \ip{x,y}_\R + \ip{x,\alpha x + \beta y}_\R \\
      & = \ip{x,y}_\R + \ip{x,z}_\R.
  \end{align*}

  So, we have proved that $\ip{\cdot,\cdot}_\R$ defines in $E$ is an inner product.
  By the free choice of $E$, on account of Lemma \ref{Theorem 1.4.5(Alsina)},
  the norm in $X$ as a real space, comes from the inner product defined
  by \eqref{eqn:real-inner-product}.

  Now, the function $\ip{x,y} := \ip{x,y}_\R + i \ip{x,iy}_\R$, for all $x,y\in X$,
  defines a complex inner product which induces the norm on $X$.

  It is interesting to note that,
  \begin{align*}
    \langle x,y \rangle
      & = \langle x, y \rangle_\R + i \langle x, iy \rangle_\R \\
      & = \frac1{\pi} \int_0^{2\pi} \cos \theta \rho_{_{+}}(x,e^{i\theta}y)
           + \frac i\pi \int_0^{2\pi}\cos \theta \rho_{_{+}}(x,ie^{i\theta}y) d\theta \\
      & = \frac1{\pi} \int_0^{2\pi} \cos \theta \rho_{_{+}}(x,e^{i\theta}y)
           + \frac i{\pi} \int_0^{2\pi} \cos \theta \rho_{_{+}}\bigl(x,e^{i(\theta+\frac\pi2)}y\bigr) d\theta \\
      & = \frac1{\pi} \int_0^{2\pi} \cos \theta \rho_{_{+}}(x,e^{i\theta}y)
           + \frac i{\pi} \int_0^{2\pi} \sin \theta \rho_{_{+}}(x,e^{i \theta}y) d\theta \\
      & = \frac1{\pi} \int_0^{2\pi} (\cos \theta +i \sin \theta)\rho_{_{+}}(x,e^{i \theta}y) d\theta = \rho_{_{\infty}}(x,y).
  \end{align*}
\end{proof}
As an immediate consequence of Theorem \ref{CIPS}, we have the following result.
\begin{cor}
For a complex normed space $X$, the following are equivalent.
  \begin{enumerate}[\upshape(i)]
    \item $\rho_{_{+}}(x,y) = \rho_{_{+}}(y,x)$ for all $x,y\in X$,
    \item $\rho_{_{\infty}}(x,y) = \rho_{_{\infty}}(y,x)$ for all $x,y\in X$.
  \end{enumerate}
\end{cor}

\section{Linear mappings preserving $\rho_{_{\infty}}$-orthogonality}
The problem of determining the structure of linear mappings between normed linear spaces,
which leave certain properties invariant, has been considered in several papers.
These are the so-called linear preserver problems.
The study of linear orthogonality preserving mappings can be considered as a part
of the theory of linear preservers, see \cite{B.T, C.L.2015, Ch4, C.W.2, K.R, K, Wo.2012} and the references therein.
\begin{dfn}
A mapping $T: X \to Y$, between complex normed spaces $X$ and $Y$, is called
$\rho_\infty$-orthogonality preserving if
\begin{equation*}
\forall_{x, y \in X}\,\, x\perp_{\rho_{_{\infty}}} y \Longrightarrow Tx\perp_{\rho_{_{\infty}}} Ty.
\end{equation*}
\end{dfn}
The aim of this section is to present results concerning the linear mappings which
preserve $\rho_{_{\infty}}$-orthogonality.
Recall that (see \cite{A.S.T, Dra}) a convex function $\phi : X\to \R$ is said to be
G\^{a}teaux differentiable at $x\in X$ if the limit
\begin{equation*}
  \phi_x(y) = \lim_{t\to0} \frac{\phi(x+ty)-\phi(x)}{t}
\end{equation*}
exists for all $y\in X$. It is well known that if $\phi$ is continuous and
G\^{a}teaux differentiable at $x$, then $\phi_x : X\to \R$, $y\mapsto \phi_x(y)$
is a bounded real linear functional. The map $\phi_x$ is called the G\^{a}teaux differential of $\phi$ at $x$.

We need the following lemmas.
\begin{lem}[{\cite[Proposition 2.1.]{B.T}}]\label{lem:unique support functional}
Let $(X, \|\!\cdot\!\|)$ be a complex normed space and let $X$ be smooth at $x\in X\setminus\{0\}$.
Then $F_x = f_x + i f_{ix}$ is the unique support functional at $x$, where $f_x(y) = \frac{\rho_{_{+}}(x,y)}{\|x\|}$ for all $y\in X$.
\end{lem}
\begin{lem}\label{lem:ker Fx}
Let $(X, \|\!\cdot\!\|)$ be a complex normed space. If $X$ is smooth at $x\in X\setminus\{0\}$
and $F_x$ is the unique support functional at $x$,
then the following are equivalent, for every $y\in X$.
  \begin{enumerate}[\upshape(i)]
    \item $y\in \ker F_x$,
    \item $\rho_{_{\infty}}(x,y)=0$.
    \end{enumerate}
\end{lem}

\begin{proof}
(i) $\Rightarrow$(ii)
Suppose $y\in\ker F_x$. Then $e^{i\theta}y\in \ker F_x$, from which we get
$\rho_{_{+}}(x,e^{i\theta}y)=0$, for all $\theta\in \R$. Therefore
\begin{equation*}
\rho_{_{\infty}} (x,y) = \frac1{\pi} \int_0^{2\pi} e^{i\theta} \rho_{_{+}}(x,e^{i\theta}y) d\theta= 0.
\end{equation*}

(ii) $\Rightarrow$ (i)
Since $X$ is smooth at $x\in X\setminus\{0\}$, we have
 $\rho_+(x,y) = \|x\| f_x(y)$, for all $y\in X$, and thus it is real linear
  on its second variable $y$. Therefore,
  \begin{equation*}
    \rho_{_{+}}(x,e^{i\theta}y) = \cos \theta \rho_{_{+}}(x,y) + \sin\theta \rho_{_{+}}(x,iy).
  \end{equation*}
Hence,
  \begin{align*}
    \rho_{_{\infty}}(x,y)
      & = \frac1{\pi} \int_0^{2\pi} (\cos\theta + i\sin\theta)\rho_{_{+}}(x,e^{i\theta}y) d\theta \\
      & = \frac1{\pi} \int_0^{2\pi} \Big(\cos^2\theta \rho_{_{+}}(x,y) + \cos\theta \sin\theta \rho_{_{+}}(x,iy)\Big) d\theta \\
      & \qquad \qquad \qquad
         + \frac{i}{\pi} \int_0^{2\pi} \Big(\sin\theta \cos \theta \rho_{_{+}}(x,y) + \sin^2\theta \rho_{_{+}}(x,iy)\Big) d\theta \\
      & = \rho_{_{+}}(x,y) +i\rho_{_{+}}(x,iy).
  \end{align*}
Since $\rho_{_{\infty}}(x,y)=0$, we must have $\rho_{_{+}}(x,y)=\rho_{_{+}}(x,iy)=0$.
Thus
  \begin{equation*}
    \|x\|F_x(y) = \rho_+(x,y) + i \rho_+(ix,y) = \rho_+(x,y) - i \rho_+(x,iy) = 0,
  \end{equation*}
  meaning that $y\in \ker F_x$.
\end{proof}
We are now in a position to prove the main result of this section.
We use some ideas of \cite[Theorem 3.1]{B.T}.
\begin{thm}\label{T41}
  Let $X$ and $Y$ be complex normed spaces and let $T:X\to Y$ be a non-zero bounded linear mapping.
  The following conditions are equivalent.
  \begin{enumerate}[\upshape(i)]
    \item $T$ preserves $\rho_{_{\infty}}$-orthogonality,
    \item $\|Tx\| = \|T\| \|x\|$, for all $x\in X$,
    \item $\rho_{_{\infty}}(Tx,Ty) = \|T\|^2 \rho_{_{\infty}}(x,y)$, for all $x,y\in X$.
  \end{enumerate}
      If $X = Y$, then each one of these assertions is also equivalent to
  \begin{enumerate}[\upshape(iv)]
    \item there exists a semi-inner product $[\cdot\,|\,\cdot]:X\times X\longrightarrow \mathbb{C}$ satisfying
\begin{align*}
[Tx,Ty] = \|T\|^2[x,y] \qquad (x, y\in X).
\end{align*}
  \end{enumerate}
\end{thm}

\begin{proof}
The implications (ii)$\Rightarrow$(iii) and (iii)$\Rightarrow$(i) are clear.

Now, suppose that (i) holds.
To prove that $T$ is
  a scalar multiple of isometry, we show that
  \begin{enumerate}[\upshape(1)]
    \item $T$ is injective; that is $Tx=0$ implies $x=0$;
    \item $T$ is an isometry; that is $\|x\|=\|y\|$ implies $\|Tx\| = \|Ty\|$.
  \end{enumerate}

  To prove (1), suppose that $Tx = 0$, for some $x\neq0$. Let $y$ be an element
  in $X$ independent of $x$, and choose a non-zero $\beta \in \R$ such that
  \begin{equation}\label{eqn:choose=beta}
    0 < \frac{\|y\|}{\|x + \beta y\|} \beta  < \frac{1}{\big(1+2\mathcal{R}(X^*)\big)}.
  \end{equation}
Let $z = x + \beta y$.
By Lemma \ref{prop:exists-a-such-that-rho-infty(x,ax+y)=0},
  $z \perp_{\rho_{_{\infty}}} (\alpha z + y)$ with $\alpha = -\frac{\overline{\rho_{_{\infty}}}(z,y)}{\|z\|^2}$.
  Since $T$ preserves $\rho_{_{\infty}}$-orthogonality, we get
  $Tz \perp_{\rho_{_{\infty}}} T(\alpha z + y)$.
  Thus $\beta Ty \perp_{\rho_{_{\infty}}} (1+\alpha\beta)Ty$.
  By Proposition \ref{prop:rho-infty(ax,ax+y)}, we have
  \begin{equation}\label{eqn:0=dots}
    0 = \rho_{_{\infty}}(\beta Ty, (1+\alpha\beta)Ty)
      = \beta (1+\bar\alpha\beta) \|Ty\|^2.
  \end{equation}
Using Theorem \ref{thm:Cauchy-Sw-inequality}, from \eqref{eqn:choose=beta}, we have
  \begin{equation*}
    |\bar\alpha \beta| \leq \big(1+2\mathcal{R}(X^*)\big)\frac{\|y\|}{\|z\|} \beta < 1.
  \end{equation*}
Therefore, $1+ \bar \alpha \beta \neq 0$. This, together with \eqref{eqn:0=dots}, yields that
  $Ty=0$, for all $y$ independent of $x$ and hence $T=0$. This contradiction proves (1).

  Now, we prove that $\|x\|=\|y\|$ implies $\|Tx\|=\|Ty\|$.
  This is clear if $x$ and $y$ are linearly dependent.
  Suppose that $x$ and $y$ are linearly independent.
  Let $M$ be the complex linear subspace of $X$
  generated by $x,y$. For $u\in M$, define ${\|u\|}_T := \|Tu\|$. Since $T$ is injective,
  ${\|\!\cdot\!\|}_T$ is a norm on $M$. Let $\Delta$ be the set of those points $u\in M$ at which
  at least one of the norms, $\|\cdot \|$ or $\enorm_T$, is not Gateaux differentiable. For
  $u\in M \setminus \Delta$, let $F_u$ and $G_u$ denote the unique support functionals at $u$
  with respect to, respectively, $\|\!\cdot\!\|$ and ${\|\!\cdot\!\|}_T$. Let $v\in \ker F_u$.
  Since $(M, \|\!\cdot\!\|)$ is smooth at $u$, by Lemma \ref{lem:ker Fx}, $\rho_{_{\infty}}(u,v)=0$.
  Hence $\rho_{_{\infty}}(Tu, Tv)=0$. Since $(M, {\|\!\cdot\!\|}_T)$ is smooth at $u$, again by Lemma \ref{lem:ker Fx},
  we get $u\in \ker G_u$. So, $\ker F_u \subset \ker G_u$.
  The rest of the proof is similar to that of \cite[Theorem 3.1]{B.T}.

  Now, we show (ii)$\Rightarrow$(iv). If (ii) holds, then the mapping
  $U=\frac{T}{\|T\|}:X\longrightarrow X$ is an isometry on $X$.
  By \cite[Theorem 1]{K.R}, there exists a semi-inner product
  $[\cdot,\cdot]\,:X\times X\longrightarrow \mathbb{C}$ such that $[Ux, Uy] = [x, y]$ for all $x, y\in X$.
  Thus $[Tx, Ty] = \|T\|^2[x, y]$ for all $x, y\in X$.

   (iv)$\Rightarrow$(ii) Trivial.
\end{proof}

Finally, taking $X = Y$ and $T = id$, the identity map of $X$,
one obtains, from Theorem \ref{T41}, the following result.

\begin{cor}
Let $X$ be a complex space endowed with two norms ${\|\!\cdot\!\|}_1$ and ${\|\!\cdot\!\|}_2$,
which generate respective functionals $\rho_{_{\infty, 1}}$ and $\rho_{_{\infty, 2}}$.
Then the following conditions are equivalent:
   \begin{enumerate}[\upshape(i)]

      \item there exist constants $0 < m \leq M$ such that
          \begin{equation*}
           m|\rho_{_{\infty, 1}}(x, y)|
               \leq |\rho_{_{\infty, 2}}(x, y)|
                \leq M |\rho_{_{\infty, 1}}(x, y)| \quad (x, y\in X),
          \end{equation*}

      \item the spaces $(X, {\|\!\cdot\!\|}_1)$ and $(X, {\|\!\cdot\!\|}_2)$ are isometrically isomorphic.
   \end{enumerate}
\end{cor}

\bibliographystyle{amsplain}

\begin{thebibliography}{99}

\bibitem{Alonso.Martini.Wu}
J.~Alonso, H.~Martini and S.~Wu,
\textit{Orthogonality types in normed linear spaces},
Chapter 4 of Surveys in Geometry,
Ed.~A.~Papadopoulos, pp. 97--170, Springer, Cham, 2022.

\bibitem{A.S.T} C.~Alsina, J.~Sikorska and M.~S.~Tom\'{a}s,
\textit{Norm Derivatives and Characterizations of Inner Product Spaces},
World Scientific, Hackensack, NJ, 2010.

\bibitem{Amir} D.~Amir,
\textit{Characterizations of inner products spaces},
Birkauser Verlag, Basel (1986).

\bibitem{B} G.~Birkhoff,
\textit{Orthogonality in linear metric spaces},
Duke Math.~J. \textbf{1} (1935), 169--172.

\bibitem{B.T} A.~Blanco and A.~Turn\v{s}ek,
\textit{On maps that preserves orthogonality in normed spaces},
Proc.~Roy.~Soc.~Edinburgh Sect.~A \textbf{136} (2006), no.~4, 709--716.

\bibitem{C.L.2015} Ch.~Chen and F.~Lu,
\textit{Linear maps preserving orthogonality},
Ann.~Funct.~Anal. \textbf{6} (2015), no.~4, 70--76.

\bibitem{Ch4} J.~Chmieli\'{n}ski,
\textit{Remarks on orthogonality preserving mappings in normed spaces and some stability problems},
Banach J.~Math.~Anal. \textbf{1} (2007), no. 1, 117--124.

\bibitem{Ch-Sik-Wo-2020} J.~Chmieli\'nski, J.~Sikorska, and P.~W\'ojcik,
\textit{On a $\rho$-orthogonally additive mappings},
Results Math. \textbf{75} (2020), no.~3, 1--17.

\bibitem{C.W.1} J.~Chmieli\'{n}ski and P.~W\'{o}jcik,
\textit{On a $\rho$-orthogonality},
Aequat.~Math. \textbf{80} (2010), 45--55.

\bibitem{C.W.2} J.~Chmieli\'{n}ski and P.~W\'{o}jcik,
\textit{$\rho$-orthogonality and its preservation-revisited},
In: Recent Developments in Functional Equations and Inequalities, vol.~99,
Banach Center Publ. (2013), 17--30.


\bibitem{D.A.J} M.~Dehghani, M.~Abed and R.~Jahanipur,
\textit{A generalized orthogonality relation via norm derivatives in real normed linear spaces},
Aequat.~Math. \textbf{93} (2019), 651--667.

\bibitem{Dra} S.~S.~Dragomir,
\textit{Semi-Inner Products and Applications},
Nova Science Publishers Inc, Hauppauge (2004).

\bibitem{giles} J.~R.~Giles,
\textit{Classes of semi--inner--product spaces},
Trans.~Amer.~Math.~Soc., \textbf{129} (1967), 436--446.

\bibitem{J.1} R.~C.~James,
\textit{Orthogonality and linear functionals in normed linear spaces},
Trans.~Amer.~Math.~Soc., \textbf{61} (1947), 265--292.

\bibitem{J.2} R.~C.~James,
\textit{Inner product in normed linear spaces},
Bull.~Am.~Math.~Soc., \textbf{53} (1947), 559--566.

\bibitem{Ku.Sa} D.~Khurana and D.~Sain,
\textit{Norm derivatives and geometry of bilinear operators},
Ann.~Funct.~Anal. \textbf{12}, 49 (2021).

\bibitem{K.R} D.~Koehler and P.~Rosenthal,
\textit{On isometries of normed linear spaces},
Studia Math. \textbf{36} (1970), 213--216.

\bibitem{K} A.~Koldobsky,
\textit{Operators preserving orthogonality are isometries},
Proc.~Roy.~Soc.~Edinburgh Sect.~A. \textbf{123} (1993), no.~5, 835--837.

\bibitem{Lumer} G.~Lumer,
\textit{Semi-inner product spaces},
Trans.~Am.~Math.~Soc., \textbf{100} (1961), 29--43.

\bibitem{Milicic-1987} P.~M.~Mili\v{c}i\'c,
\textit{Sur la G-orthogonalit\'{e} dans les esp\'{e}ace\'{e}s norm\'{e}s},
Math.~Vesnik. \textbf{39} (1987), 325--334.

\bibitem{M.M.P.S} M.~S.~Moslehian, G.~A.~Mu\~{n}oz-Fern\'{a}ndez, A.~M.~Peralta and J.~B.~Seoane-Sep\'{u}lveda,
\textit{Similarities and differences between real and complex Banach spaces: an overview and recent developments},
Rev.~Real Acad.~Cienc.~Exactas F\`{\i}s.~Nat.~Ser.~A-Mat. (RACSAM), \textbf{116}, 88 (2022).

\bibitem{M.Z.D} M.~S.~Moslehian, A.~Zamani and M.~Dehghani,
\textit{Characterizations of smooth spaces by $\rho_*$-orthogonality},
Houston J.~Math. \textbf{43} (2017), no.~4, 1187--1208.

\bibitem{Sain 2021} D.~Sain,
\textit{Orthogonality and smoothness induced by the norm derivatives},
Rev.~Real Acad.~Cienc.~Exactas F\`{\i}s.~Nat.~Ser.~A-Mat. (RACSAM), \textbf{115}, 120 (2021).

\bibitem{St.Wo} T.~Stypu{\l}a and P.~W\'{o}jcik,
\textit{Characterizations of rotundity and smoothness by approximate orthogonalities},
Ann.~Math.~Sil. \textbf{30} (2016), 183--201.

\bibitem{Wo.2012} P.~W\'{o}jcik,
\textit{Linear mappings preserving $\rho$-orthogonality},
J.~Math.~Anal.~Appl. \textbf{386} (2012), 171--176.

\bibitem{Wo.2016} P.~W\'{o}jcik,
\textit{Self-adjoint operators on real Banach spaces},
Nonlinear Anal. TMA \textbf{81}, (2013), 54--61.

\bibitem{Wo.LAA.2019} P.~W\'{o}jcik,
\textit{Characterization of linear similarities through functional equation and mappings preserving orthogonalities}
Linear Algebra Appl. \textbf{579}, (2019), 206--216.

\bibitem{W.Z.L} P.~W\'{o}jcik and A.~Zamani,
\textit{From norm derivatives to orthogonalities in Hilbert $C^*$-modules},
to appear in Linear Multilinear Algebra.

\bibitem{Z.D} A.~Zamani and M.~Dehghani,
\textit{On exact and approximate orthogonalities based on norm derivatives},
In: Chapter \textbf{21}, Ulam Tye Stability, Springer, (2019), 469--507.

\bibitem{Z.M} A.~Zamani and M.~S.~Moslehian,
\textit{An extension of orthogonality relations based on norm derivatives},
Q.~J.~Math. \textbf{70} (2019), no.~2, 379--393.

\end{thebibliography}

\end{document}